\documentclass[reqno]{amsart}
\usepackage{amsfonts}
\usepackage{amssymb}
\usepackage{color}

\theoremstyle{plain}
\newtheorem{thm}{Theorem}

\newtheorem{lem}{Lemma}
\newtheorem{proposition}{Proposition}

\theoremstyle{definition}
\newtheorem{defn}{Definition}
\theoremstyle{remark}
\newtheorem{rem}{Remark}
\numberwithin{equation}{section}

\begin{document}
\title{Quasisymmetric minimality on packing dimension for homogeneous perfect sets}
\author{Shishuang Liu}
\address[Shishuang Liu]{College of Mathematics and Information Science, Guangxi University, Nanning, 530004, P.~R. China}
\email{lssnlb@163.com}

\author{Yanzhe Li$^{*}$}
\address[Yanzhe Li]{College of Mathematics and Information Science, Guangxi University, Nanning, 530004, P.~R. China}
\email{lyzkbm@163.com}

\author{Jiaojiao Yang}
\address[Jiaojiao Yang]{School of Mathematics and Statistics, Anhui Normal University, Wuhu, 241002, P.~R. China}
\email{y.jiaojiao1025@ahnu.edu.cn}

\thanks{This work is supported by National Natural Science Foundation of China (Grant No. 11901121 and Grant No.11801199), Guangxi Natural Science Foundation (2020GXNSFAA297040) and Natural Science Foundation of Anhui Province (CN)(1908085QA30). }
\thanks{*Corresponding author. }

\begin{abstract}
In this paper, we study the quasisymmetric packing minimality of homogeneous perfect sets, and obtain that a special class of homogeneous perfect sets  with $\operatorname{dim}_{P}E=1$ is quasisymmetrically packing minimal.

\end{abstract}

\maketitle

\section{Introduction}
\label{sec:intro}

	Let $f:\mathbb{R}^{n}\to\mathbb{R}^{n}$ be a homeomorphism, if there exists a homeomorphism $\eta:[0,\infty)\to \mathbb{R}^{n}$, such that any triple points $a,b,x\in\mathbb{R}^{n}$ satisfy
\begin{equation*}
  \frac{\left|f(x)-f(a)\right|}{\left|f(x)-f(b)\right|}\le\eta(\frac{\left|x-a\right|}{\left|x-b\right|}),
\end{equation*}
then $f$ is called a $n-$dimensional quasisymmetric mapping. The quasisymmetric mappings contain the Lipschitz mappings, but the fractal dimensions of the fractal sets may not invariant under the quasisymmetric mappings, where the Lipschitz mappings preserve the fractal dimensions. How the quasisymmetric mappings change the fractal dimensions  has always been a hot research topic in the cross-over study of the fractal geometry and the quasisymmetric mappings,  quasisymmetrically minimal sets are  very important research objects in this subject. Suppose $E\subset\mathbb{R}^{n}$, then $E$ is called a quasisymmetrically Hausdorff minimal set if for any $n-$dimensional quasisymmetric mapping $f$,  we have $\dim_{H}f(E)\ge \dim_{H}E$. Similarly, we can define quasisymmetrically packing minimal set by packing dimension.

 In recent years, scholars  have done a lot of research on quasisymmetrically minimal sets. There are some typical results on quasisymmetric Hausdorff minimality, see in \cite{BCJ,TJ,GFW1,GFW2,KLV,HHA,HW,DWX,WW,XYQ}. Compared with quasisymmetric Hausdorff minimality, there are few results on quasisymmetric packing minimality. Kovalev\cite{KLV} showed that if $E\subset \mathbb{R}$ is a  quasisymmetrically packing minimal set, then $\dim_{P}E=0$ or $\dim_{P}E=1$. It is obvious that any set with packing dimension 0 is a quasisymmetrically packing minimal set, then we focus on the quasisymmetric packing minimality on the sets in $\mathbb{R}$ with packing dimension 1. Li, Wu and Xi\cite{LWX} obtained that two large classes of Moran sets in $\mathbb{R}$  with packing dimension 1 are quasisymmetrically packing minimal. Wang and Wen\cite{WW} proved that all uniform Cantor sets with packing dimension 1 are quasisymmetrically packing minimal.

A large class of the homogeneous perfect sets with packing dimension 1 is proved to be quasisymmetrically packing minimal in this paper, which generalizes a result in \cite{WW}.

This paper is organized as follows. In Section 2, we recall the definition of homogeneous perfect sets  and give some lemmas. In Section 3, we state our main result. The proof of our main result is given in Section 4 and Section 5.
\bigskip

\section{Preliminaries}
\label{sec:pre}
\subsection{Homogeneous Perfect Sets}
Wen and Wu gave the definition of the homogeneous perfect sets in \cite{WW05}, now we recall it.

Let the  sequences $\left\{c_{k}\right\}_{k\ge1}\subset \mathbb{R}^{+}$ and $\left\{n_{k}\right\}_{k\ge1}\subset \mathbb{Z}^{+}$ with $n_{k}\ge 2$ and $n_{k}c_{k}< 1$ for any $k\ge1$. For any $k\ge1$,
let $D_{k}=\left\{i_{1}i_{2}\cdots i_{k}:1\le i_{j} \le n_{j},1\le j \le k\right\}$, $D_{0}=\emptyset$ and $D=\cup_{k\ge0}D_{k}$. If $\sigma=\sigma_{1}\sigma_{2}\cdots\sigma_{k}\in D_{k}$, $\tau=\tau_{1}\tau_{2}\cdots\tau_{m}$, where $1\le \tau_{i}\le n_{k+i}$ for any $1\le i\le m$, let $\sigma*\tau =\sigma_{1}\sigma_{2}\cdots\sigma_{k}\tau_{1}\tau_{2}\cdots\tau_{m}\in D_{k+m}$.

\begin{defn} \label{HPS}(Homogeneous perfect sets \cite{WW05})

For a closed interval  $I_{0}\subset\mathbb{R}$ with $I_{0} \neq \emptyset$, which we call the initial interval, we say the collection of closed subintervals  $\mathcal{I}=\left\{I_{\sigma}:\sigma\in{D}\right\}$ of $I_{0}$ has homogeneous perfect structure if it satisfies:
\begin{enumerate}
\item[\textup{(1)}] $I_{\emptyset}=I_{0}$;
\item[\textup{(2)}] for any $k\ge1$ and $\sigma\in{D_{k-1}}$, $I_{\sigma * 1}, \cdots, I_{\sigma * n_{k}}$ are closed subintervals of $I_{\sigma}$  with $\min(I_{\sigma*(l+1)})\ge \max(I_{\sigma*l})$ for any $1\le l \le n_{k}-1$;
\item[\textup{(3)}] for any $k\ge1$ and $\sigma\in{D_{k-1}}$, $1\le j \le n_{k}$,$$\frac{\left|I_{\sigma * j}\right|}{\left|I_{\sigma}\right|}=c_{k},$$
where the diameter of the set $A$ is denoted by $\left|A\right|$;
\item[\textup{(4)}]  there exists a sequence  $\{\xi_{k,l}\}_{k\ge1,\ 0\le l\le n_k}\subset \mathbb{R}^{+}\cup \{0\}$ satisfying for any $k\ge1$ and $\sigma\in{D_{k-1}}$,
$$\min(I_{\sigma*1})-\min(I_{\sigma})=\xi_{k,0};$$ $$\min(I_{\sigma*(l+1)})-\max(I_{\sigma*l})=\xi_{k,l}(\forall 1\le l \le n_{k}-1);$$
		$$\max(I_{\sigma})-\max(I_{\sigma*n_{k}})=\xi_{k,n_{k}}.$$
\end{enumerate}

If $\mathcal{I}$ has homogeneous perfect structure, let $E_{k}=\cup_{\sigma\in{D_{k}}}I_{\sigma}$ for any $k\ge0$, , then
	$E=\cap_{k\ge0}E_{k}=E(I_{0},\left\{n_{k}\right\},\left\{c_{k}\right\},\left\{\xi_{k,l}\right\})$ is called a homogeneous perfect set.
	For any $k\ge0$, let $\mathcal{E}_{k}=\left\{I_{\sigma}:\sigma\in{D_{k}}\right\}$, then any $I_{\sigma}$ in $\mathcal{E}_{k}$ is called a $k$-order basic interval of $E$.
\end{defn}

\begin{rem}
If $E=E(I_{0},\left\{n_{k}\right\},\left\{c_{k}\right\},\left\{\xi_{k,l}\right\})$ with $\xi_{k,0}=\xi_{k,n_{k}}=0$, $\xi_{k,i}=\xi_{k,j}$ for any $k\ge 1$ and $1\le i<j\le n_{k}$, then $E$ is a uniform Cantor set(see in \cite{WW}).
\end{rem}


\bigskip


\subsection{Some Lemmas}
We need the follwing lemmas to finish our proof.

The next lemma gives the packing dimension formula for some homogeneous perfect sets.
\begin{lem}\label{lem1}\cite{WXY}
Suppose that
	$E=E(I_{0},\left\{n_{k}\right\},\left\{c_{k}\right\},\left\{\xi_{k,l}\right\})$, and there exists a real number $\chi\ge 1$ such that $\max_{ 1 \leq l \leq n_{k}-1}\xi_{k,l}\le \chi\min_{ 1 \leq l \leq n_{k}-1}\xi_{k,l}$ for any $k\ge 1$, then
	\begin{equation*}
	\operatorname{dim}_{P}E=\varlimsup_{k \rightarrow \infty}\frac{\log n_{1}\cdots n_{k}n_{k+1}}{-\log\frac{c_{1}c_{2}\cdots c_{k}-\xi_{k+1, 0}-\xi_{k+1, n_{k+1}}}{n_{k+1}}}.
	\end{equation*}
\end{lem}
\bigskip

We need the mass distribution principle to estimate the lower bound of the packing dimension of the sets.
\begin{lem}\label{lem2}\textmd{\rm (Mass distribution principle \cite{WZY})}
Suppose that $s\ge0$, let $\mu$ be a Borel probability measure on  a Borel set $E\subseteq\mathbb{R}$. If there is a positive constant $C$, such that
	\begin{equation*}
	\varliminf_{r \rightarrow 0}\frac{\mu(B(x,r))}{r^{s}}\le C\
	\end{equation*}
for any	$x\in E$, then $\operatorname{dim}_{P}E\ge s.$
\end{lem}
\bigskip

The following lemma shows some  relationships  between the lengths for the image sets of the quasissymmertic mappings and the lengths for the original sets.
\begin{lem}\label{lem3}\cite{wu93}
Let $f:\mathbb{R}\to \mathbb{R}$ be a \text{\rm1}-dimensional quasisymmetric mapping, then there exist positive real numbers $\beta>0$, $K_{\rho}>0$ and $0<p\le 1\le q$ such that
	$$\beta(\frac{\left|I^{'}\right|}{\left|I\right|})^{q}\le \frac{\left|f(I^{'})\right|}{\left|f(I)\right|}\le 4(\frac{\left|I^{'}\right|}{\left|I\right|})^{p},$$
	where $I^{'}$ and $I$ are any  intervals satisfying $I^{'}\subseteq I$ ,
and
$$\frac{\left|f(\rho I)\right|}{\left|f(I)\right|}\le K_{\rho},$$
where for any $\rho>0$, $\rho I$ denotes the   interval with the same center of the interval $I$, and $|\rho I|=\rho |I|$.
\end{lem}

\bigskip
\section{Main result}
Our main result is stated as follows.
\begin{thm}\label{thm}
	Suppose
	$E=E(I_{0},\left\{n_{k}\right\},\left\{c_{k}\right\},\left\{\xi_{k,l}\right\})$, and there exists a real number $\chi\ge 1$, such that $\max_{ 1 \leq l \leq n_{k}-1}\xi_{k,l}\le \chi\min_{ 1 \leq l \leq n_{k}-1}\xi_{k,l}$ for any $k\ge 1$. If $\operatorname{dim}_{P}E=1$, then for any \text{\rm1}-dimensional quasisymmetric mapping $f$,  we have $\operatorname{dim}_{P}f(E)=1$.
\end{thm}

\begin{rem}
	In  Theorem \text{\rm1.2} of \text{\rm\cite{WW}}, Wang and Wen proved that  for any  uniform Cantor sets $E$,  if $\operatorname{dim}_{P}E=1$, then we have $\operatorname{dim}_{P}f(E)=1$ for any \text{\rm1}-dimensional quasisymmetric mapping $f$ .
	It is obvious that the homogeneous perfect sets satisfying the condition of Theorem  \text{\rm\ref{thm}} in this paper contain the uniform Cantor sets. Thus  Theorem \text{\rm\ref{thm}} in this paper generalizes  Theorem \text{\rm1.2} in \text{\rm\cite{WW}}.
\end{rem}

\bigskip

\section{The reconstruction of Homogeneous perfect sets}

In order to prove the Theorem \ref{thm}, we reconstruct the homogeneous perfect set $E=E(I_{0},\left\{n_{k}\right\},\left\{c_{k}\right\},\left\{\xi_{k,l}\right\})$ and represent it as an equivalent form which is easier to discuss in our proof.

For any $k\ge 0$, $\sigma\in D_{k}$, let $I_{\sigma}^{*}$ be a closed subinterval of $I_{\sigma}$ satisfying the following conditions:
\begin{enumerate}
	\item[\textup{(A)}]$\min(I_{\sigma}^{*})-\min(I_{\sigma})=\xi_{k+1, 0}$,  $\max(I_{\sigma})-\max(I_{\sigma}^{*})=\xi_{k+1, n_{k+1}}$;
	\item[\textup{(B)}]$\left|I_{\sigma}^{*}\right|=\sum_{l=1}^{n_{k+1}-1}\xi_{k+1,l}+c_{1}c_{2}\cdots c_{k+1}n_{k+1}$.
\end{enumerate}

 Let $I_{0}^{*}=I_{\emptyset}^{*}$, denote $\delta_{k}=\left|I_{\sigma}^{*}\right|$, $\delta_{0}=\left|I_{0}^{*}\right|$ for any $ k\ge 1$ and $\sigma\in D_k$. Suppose that $E_{k}^{*}=\cup_{\sigma\in{D_{k}}}I_{\sigma}^{*}$ for any $k\ge 0$ and $\sigma\in D_k$, then it obvious that $E=\cap_{k\ge0}E_{k}=\cap_{k\ge0}E_{k}^{*}$.

In fact, $E=E(I_{0}^{*},\left\{n_{k}^{*}\right\},\left\{c_{k}^{*}\right\},\{\xi_{k,l}^{*}\})$ is a homogeneous perfect set with the following parameters for any $k\ge 1$ , where $\mathcal{I}^{*}=\{I_{\omega}^{*}:\omega\in D\}$ has homogeneous perfect structrue:
\begin{enumerate}
	\item[\textup{(1)}]$I_{0}^{*}=I_{0}-[\min(I_{0}),\min(I_{0})+\xi_{1, 0})-(\max(I_{0})-\xi_{1, n_{1}},\max(I_{0})]$;
	\item[\textup{(2)}]$c_{k}^{*}=\frac{\delta_{k}}{\delta_{k-1}}$, $n_{k}^{*}=n_{k}$;
	\item[\textup{(3)}]$\xi_{k,l}^{*}=\xi_{k,l}+\xi_{k+1, 0}+\xi_{k+1, n_{k+1}}(\forall 1\le l\le n_{k}-1)$, $\xi_{k,0}^{*}=\xi_{k+1, 0}$, $\xi_{k, n_{k}}^{*}=\xi_{k+1, n_{k+1}}$.
\end{enumerate}

 For any $k\ge1$, denote
 $$N_{k}^{*}=n_{1}^{*}n_{2}^{*}\cdots n_{k}^{*},\ \delta_{k}^{*}=\delta_{0}c_{1}^{*}c_{2}^{*}\cdots c_{k}^{*}.$$
 $$\underline{\alpha}_{k}^{*}=\min_{ 1 \leq l \leq n_{k}-1}\xi_{k,l}^{*},\ \overline{\alpha}_{k}^{*}=\max_{ 1 \leq l \leq n_{k}-1}\xi_{k,l}^{*}.$$

 It is easy to obtain that
\begin{equation}\label{rs1}
\xi_{k,0}^{*}+\xi_{k,n_{k}}^{*}\le \underline{\alpha}_{k}^{*},
\end{equation}
and if $E$ satisfies the condition of Theorem \ref{thm}, then
\begin{equation}\label{rs2}
\overline{\alpha}_{k}^{*}\le \chi\underline{\alpha}_{k}^{*},
\end{equation}
where $\chi$ is the constant in Theorem \ref{thm}.
 \bigskip

 The folllowing lemma gives a new form of the homogeneous perfect sets in Theorem \text{\rm\ref{thm}}.
\begin{lem}\label{lem4}
	Suppose $E=E(I_{0},\left\{n_{k}\right\},\left\{c_{k}\right\},\left\{\xi_{k,l}\right\})$ satisfies the condition of Theorem \text{\rm\ref{thm}}, then there exists $\left\{H_{m}\right\}_{m\ge 0}$, which is a sequence of closed sets with length decreasing, such that $E=\cap_{k\ge0}E_{k}=\cap_{k\ge0}E_{k}^{*}=\cap_{m\ge0}H_{m}$. Furthermore,  $\left\{H_{m}\right\}_{m\ge 0}$ satisfies:
	\begin{enumerate}
		\item[\textup{(1)}] For any $m\ge0$, $H_{m}$ is a union of a finite number of closed intervals whose interiors are disjoint , which are called the branches of $H_{m}$. Denote $\mathcal{H}_{m}=\{A: A ~~~~\text{is a branch of}~~~~ H_{m}\}$;
		\item[\textup{(2)}] $\left\{E_{k}^{*}\right\}_{k\ge0}\subset\left\{H_{m}\right\}_{m\ge 0}$ and $H_{m_{k}}=E_{k}^{*}$ for any $k\ge 0$;
		\item[\textup{(3)}] There exists $M \in \mathbb{Z}^{+}$ with $M>2\chi$ such that each branch of $H_{m-1}$ contains at most $M^{2}$ branches of $H_{m}$ for any $m\ge1$, where $\chi$ is the constant in Theorem \ref{thm};
		\item[\textup{(4)}] For any $m\ge0$, $\max_{I\in\mathcal{H}_{m}}\left|I\right|\le 2\chi\min_{I\in\mathcal{H}_{m}}\left|I\right|$.
	\end{enumerate}
\end{lem}
\begin{proof}
	Let $M=[2\chi]+1$. For any $k\ge 1$, let $i_{k}$ be the positive integer satisfying the following conditions: If $n_{k}^{*}\in[2,M)$, then $i_{k}=1$; If $n_{k}^{*}\in[M,+\infty)$, then $i_{k}$ is the positive integer satisfying $n_{k}^{*}\in[M^{i_{k}},M^{i_{k}+1})$. Define $m_{0}=0,\ m_{k}=\sum_{l=1}^{k}i_{l},$ then $m_k=m_{k-1}+i_k.$
	
	For any $k\ge 0$, define $H_{m_{k}}=E_{k}^{*}$  and  $\mathcal{H}_{m_{k}}=\{I_{\omega}^{\ast}:\omega\in D_{k}\}$, which means all branches of $H_{m_{k}}$ are all $k$-order basic intervals in $E_{k}^{*}$. Next,  we construct $H_{m}$ for any $k\ge 1$ and $m_{k-1}<m<m_{k}$,.

	\begin{enumerate}
		\item[\textup{(1):}] If $n_{k}^{*}\in[M,M^{2})$, then $i_{k}=1$ and $m_k=m_{k-1}+1$,  we have nothing  to do.
		\item[\textup{(2):}] If $n_{k}^{*}\in[M^{2},+\infty)$, then $i_{k}\ge 2$, and there are $a_{0}, a_{1},\cdot\cdot\cdot, a_{i_{k}-1}$, such that $a_{j}\in\left\{0,1,\cdots,M-1\right\}$  for any $j\in\left\{0,1,\cdots,i_{k}-1\right\}$ and
		\begin{equation*}
		n_{k}^{*}=a_{0}+a_{1}M+a_{2}M^{2}+\cdots+a_{i_{k}-1}M^{i_{k}-1}+M^{i_{k}}.
		\end{equation*}
	
		For any $k\ge1$ and $\sigma\in D_{k-1}$, since $H_{m_{k-1}}=E_{k-1}^{*}$, $H_{m_{k-1}}$ has $N_{k-1}^{*}$ branches and for any $I_{\sigma}^{*}\in \mathcal{H}_{m_{k-1}}$, the number of the $k$-order basic intervals in $E_{k}^{*}$ contained in $I_{\sigma}^{*}$ is $n_{k}^{*}$,  denote $I_{\sigma*1}^{*},\cdots,I_{\sigma*n_{k}^{*}}^{*}$ .

Now we   begin to construct $H_{m_{k-1}+i}$ for any  $1\le i\le i_k-1$.

Let $[T_{1},T_{2},\cdots,T_{t}]$ be the smallest closed interval containing the $t$ closed intervals $T_{1},T_{2},\cdots,T_{t}$.
\begin{enumerate}
\item[\textup{(a)}] For any $I_{\sigma}^{*}\in \mathcal{H}_{m_{k-1}}$, let $l_{1}=a_{1}+a_{2}M+\cdots+a_{i_{k}-1}M^{i_{k}-2}+M^{i_{k}-1}$, then $n_{k}^{*}=Ml_{1}+a_{0}=a_0(l_1+1)+(M-a_0)l_1$. Define

$$I_{1}^{\sigma,1}=[I_{\sigma*1}^{*},\cdots,I_{\sigma*(l_{1}+1)}^{*}],$$
$$I_{2}^{\sigma,1}=[I_{\sigma*(l_{1}+2)}^{*}\cdots I_{\sigma*(2l_{1}+2)}^{*}],$$
$$\cdots$$
$$I_{a_{0}}^{\sigma,1}=[I_{\sigma*((a_{0}-1)l_{1}+a_{0})}^{*}\cdots I_{\sigma*(a_{0}l_{1}+a_{0})}^{*}],$$
$$I_{a_{0}+1}^{\sigma,1}=[I_{\sigma*(a_{0}l_{1}+a_{0}+1)}^{*},\cdots,I_{\sigma*(a_{0}l_{1}+a_{0}+l_{1})}^{*}],$$
$$\cdots$$
$$I_{M}^{\sigma,1}=[I_{\sigma*(n_{k}^{*}+1-l_{1})}^{*},\cdots,I_{\sigma*n_{k}^{*}}^{*}].$$

Let $H_{m_{k-1}+1}=\bigcup_{\sigma\in D_{k-1}}\bigcup_{i=1}^{M}I_{i}^{\sigma,1}$, and let the $M$ closed intervals $I_{1}^{\sigma,1},\cdots,I_{M}^{\sigma,1}$ be the $M$ branches of $H_{m_{k-1}+1}$ in $I_{\sigma}^{*}$ , then each branch of $H_{m_{k-1}}$ contains $M$ branches of $H_{m_{k-1}+1}$ and  it is easy to obtain that $\max_{I\in\mathcal{H}_{m_{k-1}+1}}\left|I\right|\le 2\chi\min_{I\in\mathcal{H}_{m_{k-1}+1}}\left|I\right|$.

\item[\textup{(b)}] If $i_{k}=2$, then $m_{k}=m_{k-1}+2$, and $H_{m_{k-1}+1}$ is defined as above, $H_{m_{k-1}}=E_{k-1}^{*}$, $H_{m_{k}}=E_{k}^{*}$. This completes the construction of $H_{m_{k-1}+i}$  for any $1\le i\le i_k-1$.

\item[\textup{(c)}] if $i_{k}\ge3$, then we continue to construct  $H_{m_{k-1}+2}$. Let $l_{2}=a_{2}+a_{3}M+\cdots+a_{i_{k}-2}M^{i_{k}-3}+M^{i_{k}-2}$, then $l_{1}=Ml_{2}+a_{1}$, $n_{k}^{*}=M^{2}l_{2}+a_{1}M+a_{0}=a_0(Ml_2+a_1+1)+(M-a_0)(Ml_2+a_1)$.

For any $I^{\sigma,1}_{i}\in \mathcal{H}_{m_{k-1}+1}(\sigma\in D_{k-1}, 1\le i\le M)$, we divide our construction into the following two cases:

\textbf{(c1):} If $1\le i\le a_{0}$,  then the number of the $k$-order basic intervals contained in each $I^{\sigma,1}_{i}$ is $l_{1}+1=Ml_{2}+a_{1}+1=(l_2+1)(a_1+1)+l_2(M-a_1-1)$. Define
$$I_{i*1}^{\sigma,1}=[I_{\sigma*((i-1)l_{1}+i)}^{*},\cdots,I_{\sigma*((i-1)l_{1}+i+l_{2})}^{*}],$$
$$I_{i*2}^{\sigma,1}=[I_{\sigma*((i-1)l_{1}+i+l_{2}+1)}^{*},\cdots,I_{\sigma*((i-1)l_{1}+i+2l_{2}+1)}^{*}],$$
$$\cdots$$ $$I_{i*(a_{1}+1)}^{\sigma,1}=[I_{\sigma*((i-1)l_{1}+i+a_{1}l_{2}+a_{1})}^{*},\cdots,I_{\sigma*((i-1)l_{1}+i+(a_{1}+1)l_{2}+a_{1})}^{*}],$$
$$I_{i*(a_{1}+2)}^{\sigma,1}=[I_{\sigma*((i-1)l_{1}+i+(a_{1}+1)(l_{2}+1))}^{*},\cdots,I_{\sigma*((i-1)l_{1}+i+(a_{1}+1)(l_{2}+1)+l_{2}-1)}^{*}],$$
$$\cdots$$
$$I_{i*M}^{\sigma,1}=[I_{\sigma*(il_{1}+i+1-l_{2})}^{*},\cdots,I_{\sigma*(il_{1}+i)}^{*}].$$

\textbf{(c2):} If $a_{0}+1\le i\le M$, then the number of the $k$-order basic intervals contained in each $I^{\sigma,1}_{i}$ is $l_{1}=Ml_{2}+a_{1}=a_1(l_2+1)+(M-a_1)l_2$. Define
$$I_{i*1}^{\sigma,1}=[I_{\sigma*((i-1)l_{1}+a_{0}+1)}^{*},\cdots,I_{\sigma*((i-1)l_{1}+a_{0}+1+l_{2})}^{*}],$$
$$I_{i*2}^{\sigma,1}=[I_{\sigma*((i-1)l_{1}+a_{0}+l_{2}+2)}^{*},\cdots,I_{\sigma*((i-1)l_{1}+a_{0}+2l_{2}+2)}^{*}],$$
$$\cdots$$ $$I_{i*a_{1}}^{\sigma,1}=[I_{\sigma*((i-1)l_{1}+a_{0}+(a_{1}-1)l_{2}+a_{1})}^{*},\cdots,I_{\sigma*((i-1)l_{1}+a_{0}+a_{1}l_{2}+a_{1})}^{*}],$$
$$I_{i*(a_{1}+1)}^{\sigma,1}=[I_{\sigma*((i-1)l_{1}+a_{0}+a_{1}l_{2}+a_{1}+1)}^{*},\cdots,I_{\sigma*((i-1)l_{1}+a_{0}+a_{1}l_{2}+a_{1}+l_{2})}^{*}],$$
$$\cdots$$
$$I_{i*M}^{\sigma,1}=[I_{\sigma*(il_{1}+a_{0}+1-l_{2})}^{*},\cdots,I_{\sigma*(il_{1}+a_{0})}^{*}].$$

Let $H_{m_{k-1}+2}=\bigcup_{\sigma\in D_{k-1}}\bigcup_{i=1}^{M}\bigcup_{j=1}^{M}I_{i*j}^{\sigma,1}$, and let the $M$ closed intervals $I^{\sigma,1}_{i*1},I^{\sigma,1}_{i*2},\cdots, I^{\sigma,1}_{i*M}$  be the $M$ branches of $H_{m_{k-1}+2}$ in $I_{i}^{\sigma,1}$,  then each branch of $H_{m_{k-1}+1}$ contains $M$ branches of $H_{m_{k-1}+2}$ and  it is easy to obtain that $\max_{I\in\mathcal{H}_{m_{k-1}+2}}\left|I\right|\le 2\chi\min_{I\in\mathcal{H}_{m_{k-1}+2}}\left|I\right|$.

\item[\textup{(d)}] If $i_{k}=3$, then $m_{k}=m_{k-1}+3$, and $H_{m_{k-1}+1}$, $H_{m_{k-1}+2}$ are defined as above, $H_{m_{k-1}}=E_{k-1}^{*}$, $H_{m_{k}}=E_{k}^{*}$. This completes the construction of $H_{m_{k-1}+i}$ for any $1\le i\le i_k-1$.
\item[\textup{(e)}] If $i_k\ge 4$, then $m_k=m_{k-1}+i_k$. If $H_{m_{k-1}+i-1}(3\le i\le i_k-1)$ has been constructed, we repeat the method of the construction of $H_{m_{k-1}+i-1}$ from $H_{m_{k-1}+i-2}$ to define $H_{m_{k-1}+i}$ from $H_{m_{k-1}+i-1}$. Then $H_{m_{k-1}+1}, H_{m_{k-1}+2},\cdot\cdot\cdot H_{m_{k-1}+i_k-1}$ are defined, and we can obtain that for any $1\le i\le i_k-1$, each branch of $H_{m_{k-1}+i-1}$ contains $M$ branches of $H_{m_{k-1}+i}$ and  $\max_{I\in\mathcal{H}_{m_{k-1}+i}}\left|I\right|\le 2\chi\min_{I\in\mathcal{H}_{m_{k-1}+i}}\left|I\right|$. This completes the construction of $H_{m_{k-1}+i}$ for any $1\le i\le i_k-1$.

\item[\textup{(f)}] For any $i_{k}\ge2$, suppose $H_{m_{k-1}+i}$ has been constructed for any $1\le i\le i_{k}-1$. Notice that  each branch of $H_{m_{k-1}+i-1}\ (2\le i\le i_{k}-1)$ contains $M$ branches of $H_{m_{k-1}+i}$, thus each branch of $H_{m_{k-1}}$ contains $M^{i_{k}-1}$ branches of $H_{m_{k-1}+i_{k}-1}$. Notice that   $m_k=m_{k-1}+i_k$ and $H_{m_{k}}=E_{k}^{*}$ for any $k\ge 0$ implies that each branch of $H_{m_{k-1}}$ contains $n_k^{*}$ branches of $H_{m_{k}}$, then each branch of $H_{m_{k-1}+i_{k}-1}$ contains at most $M^{2}$ branches of $H_{m_{k}}$(otherwise, if there exists a branch of $H_{m_{k-1}+i_{k}-1}$ containing $M^{*}>M^{2}$ branches of $H_{m_{k}}$, then any branch of $H_{m_{k-1}+i_{k}-1}$ contains $M^{*}$ or $M^{*}+1$ or $M^{*}-1$ branches of $H_{m_{k}}$, which implies that $n_{k}^{*}>M^{2}\times M^{i_{k}-1}=M^{i_k+1}$, it is contrary to $M_{i_k}\le n_k^{*}<M^{i_k+1}$).
    \item[\textup{(g)}]  Since $H_{m_{k}}=E_{k}^{*}$ for any $k\ge 0$, we have $\max_{I\in\mathcal{H}_{m_{k}}}\left|I\right|= \min_{I\in\mathcal{H}_{m_{k}}}\left|I\right|$ for any $k \ge 0$.
\end{enumerate}
	\end{enumerate}

We finish the construction of $\{H_m\}_{m\ge0}$ which satisfies $(1)-(4)$ of Lemma \ref{lem4}.
\end{proof}
\bigskip

\begin{rem}
Without loss of generality, we  assume that $I_{0}^{*}=[0,1]$, then $H_{m_{0}}=E_{0}^{*}=[0,1]$ and $\delta_{0}=1$.
\end{rem}
\bigskip

\begin{lem}\label{lem5}
	Let $E=E(I_{0},\left\{n_{k}\right\},\left\{c_{k}\right\},\left\{\xi_{k,l}\right\})$ satisfies the condition of Theorem \text{\rm\ref{thm}}, $\left\{H_{m}\right\}_{m\ge0}$ be the length decreasing sequence  in Lemma \text{\rm\ref{lem4}} and the total length of all branches of each $H_{m}$ is denoted by $l(H_{m})$, then for any $k\ge 1$ and $m_{k-1}<m<m_{k}$,
    \begin{equation}
    l(H_{m_{k}})=N_{k}^{*}\delta_{k}^{*},\ (1-\frac{2\chi}{M})N_{k-1}^{*}\delta_{k-1}^{*}\le l(H_{m})\le N_{k-1}^{*}\delta_{k-1}^{*}\  .
    \end{equation}

\end{lem}
\begin{proof}
	Since $H_{m_{k}}=E_{k}^{*}$  for any $k\ge 1$ and $\left\{H_{m}\right\}_{m\ge0}$ is a length decreasing sequence, it is obvious that $l(H_{m_{k}})=l(E_{k}^{*})=N_{k}^{*}\delta_{k}^{*}$ and
	$l(H_{m})\le l(H_{m_{k-1}})=l(E_{k-1}^{*})=N_{k-1}^{*}\delta_{k-1}^{*}$ for any $k\ge 1$ and $m_{k-1}<m<m_{k}$. Then we only need to verify that $(1-\frac{2\chi}{M})N_{k-1}^{*}\delta_{k-1}^{*}\le l(H_{m})$ for any $k\ge 1$ and $m_{k-1}<m<m_{k}$.

	We can see a fact from the construction of $\left\{H_{m}\right\}_{m\ge0}$: In order to get $H_{m_{k}-1}$, we remove a half open and half closed interval of length $\xi_{k,0}^{*}$ and a half open and half closed interval of length  $\xi_{k, n_{k}}^{*}$ from each branch of $H_{m_{k-1}}$, and remove $[\sum_{j=0}^{i_{k}-2}M^{j}(M-1)]N_{k-1}^{*}=(M^{i_{k}-1}-1)N_{k-1}^{*}$ open intervals which the lengths of them are at most $\overline{\alpha}_{k}^{*}$ from  $E_{k-1}^{*}=H_{m_{k-1}}$. Then by (\ref{rs1}), we have
	\begin{equation*}
	\begin{aligned}
	l(H_{m_{k}-1})
	&\ge N_{k-1}^{*}\delta_{k-1}^{*}-N_{k-1}^{*}(\xi_{k,0}^{*}+\xi_{k, n_{k}}^{*})-(M^{i_{k}-1}-1)N_{k-1}^{*}\overline{\alpha}_{k}^{*}\\
	&\ge N_{k-1}^{*}\delta_{k-1}^{*}-N_{k-1}^{*}\underline{\alpha}_{k}^{*}-(M^{i_{k}-1}-1)N_{k-1}^{*}\overline{\alpha}_{k}^{*}\\
	&\ge N_{k-1}^{*}\delta_{k-1}^{*}-N_{k-1}^{*}\overline{\alpha}_{k}^{*}-(M^{i_{k}-1}-1)N_{k-1}^{*}\overline{\alpha}_{k}^{*}\\
	&\ge N_{k-1}^{*}\delta_{k-1}^{*}-M^{i_{k}-1}N_{k-1}^{*}\overline{\alpha}_{k}^{*}.\\
	\end{aligned}
	\end{equation*}
	Notice that $n_{k}^{*}\ge 2$ and
	\begin{equation*}
	M^{i_{k}}\le n_{k}^{*}< M^{i_{k}+1},
	\end{equation*}
then by (\ref{rs2}), we have
	\begin{equation*}
	\begin{aligned}
	l(H_{m_{k}-1})
	&\ge N_{k-1}^{*}\delta_{k-1}^{*}-\frac{n_{k}^{*}}{M}N_{k-1}^{*}\overline{\alpha}_{k}^{*}\\
	&\ge N_{k-1}^{*}\delta_{k-1}^{*}-\frac{2(n_{k}^{*}-1)}{M}N_{k-1}^{*}\overline{\alpha}_{k}^{*}\\
	&\ge 	N_{k-1}^{*}\delta_{k-1}^{*}-\frac{2}{M}N_{k-1}^{*}\chi(n_{k}^{*}-1)\underline{\alpha}_{k}^{*}\\
	&\ge N_{k-1}^{*}\delta_{k-1}^{*}-\frac{2\chi}{M}N_{k-1}^{*}\delta_{k-1}^{*}\\
	&\ge (1-\frac{2\chi}{M})N_{k-1}^{*}\delta_{k-1}^{*}.
	\end{aligned}
	\end{equation*}

Notice that $\left\{H_{m}\right\}_{m\ge0}$ is a length decreasing sequence, then $$ l(H_{m})\ge l(H_{m_{k}-1})\ge (1-\frac{2\chi}{M})N_{k-1}^{*}\delta_{k-1}^{*}$$ for any $k\ge 1$ and $m_{k-1}<m<m_{k}$.
\end{proof}

\bigskip

\section{The Quasisymmetric packing minimality on homogeneous perfect sets}

\subsection{The measure $\mu_{d}$}

Let $E=E(I_{0},\left\{n_{k}\right\},\left\{c_{k}\right\},\left\{\xi_{k,l}\right\})$ satisfies the condition of Theorem \ref{thm}, and $f$ be a \text{\rm1}-dimensional quasisymmetric mapping, $\left\{H_{m}\right\}_{m\ge0}$ be the length decreasing sequence in Lemma \ref{lem4}. In order to complete the proof of Theorem \ref{thm} by Lemma \ref{lem2}, we need to define a positive and finite Borel measure on $f(E)$.

For any $m\ge0$, let $J_{m}$ be the image set of a branch of $H_m$ under $f$, it is obvious that image sets of all branches of $H_m$ under $f$ constitute $f(H_m)$, we call $J_m$ a branch of $f(H_m)$. Let $J_{m,1}\cdots,J_{m,N(J_{m})}$ be all branches of $f(H_{m+1})$ contained in $J_{m}$, where $N(J_{m})$ is the number of the branches of $f(H_{m+1})$ contained in $J_{m}$, then $N(J_{m})\le M^{2}$.
For any $d\in(0,1)$, $m\ge0$ and $1\le i\le N(J_{m})$, by the measure extension theorem, there is a probability Borel measure $\mu_{d}$ on $f(E)$ satisfying
\begin{equation}
\mu_{d}(J_{m,i})=\frac{\left|J_{m,i}\right|^{d}}{\sum_{j=1}^{N(J_{m})}\left|J_{m,j}\right|^{d}}\mu_{d}(J_{m}).
\end{equation}
\bigskip

For any  $m\ge1$, let $k$  satisfies $m_{k-1}< m\le m_{k}$ , denote
$$\Lambda(m)=\max_{I\in\mathcal{H}_{m-1},~~J\in\mathcal{H}_{m},J\subset I}\frac{\left|J\right|}{\left|I\right|}, \lambda(m)=\min_{I\in\mathcal{H}_{m-1},J\in\mathcal{H}_{m},J\subset I}\frac{\left|J\right|}{\left|I\right|};$$
$$\Gamma(m)=\frac{\overline{\alpha}_{k}^{*}}{\min_{I\in\mathcal{H}_{m-1}}\left|I\right|},\ \gamma(m)=\frac{\underline{\alpha}_{k}^{*}}{\max_{I\in\mathcal{H}_{m-1}}\left|I\right|}.$$
By Lemma \ref{lem4}, we have
$$\lambda(m)\le\Lambda(m)\le 4\chi^{2}\lambda(m),\ \gamma(m)\le\Gamma(m)\le 2\chi^{2}\gamma(m),$$
and
\begin{equation}\label{52}
l(H_{m})\le \prod_{i=1}^{m}M^{2}\Lambda(i),\ l(H_{m})\le \prod_{i=1}^{m}(1-\gamma(i)).
\end{equation}
\bigskip

\bigskip

We have the following results.
\begin{lem}\label{lem6}
	If $\operatorname{dim}_{P}E=1$, then there exists a subsequence $\left\{a_{k}\right\}_{k\ge0}$ of  $\left\{m_{k}-1\right\}_{k\ge0}$, such that
	\begin{enumerate}
		\item[\textup{(1)}] $$\lim_{k \rightarrow\infty}(l(H_{a_{k}}))^{\frac{1}{a_{k}}}=1;$$
		\item[\textup{(2)}] Let $S_{\varepsilon}(m)=\left\{j:1\le j\le m, \gamma(j)\le \varepsilon\right\}$ for any $\varepsilon\in(0,1)$ and $m\ge1$, then
		\begin{equation*}
		\lim_{k \rightarrow\infty}\frac{\#S_{\varepsilon}(a_{k})}{a_{k}}=1,
		\end{equation*}
where the cardinality is denoted by $\#$;
		\item[\textup{(3)}] For any $p\in(0,1]$, we have
		\begin{equation}\label{53}
		\lim_{k \rightarrow\infty}(\prod_{j\in S_{\varepsilon}(a_{k})}(1-(\gamma(j))^{p}))^{\frac{1}{a_{k}}}=1
		\end{equation}
	for sufficiently small $\varepsilon\in(0,1)$.
	\end{enumerate}
\end{lem}
\begin{proof}
	(1) By $\operatorname{dim}_{P}E=1$ and Lemma \ref{lem1}, we have
	\begin{equation*}
	\varlimsup_{k \rightarrow \infty}\frac{\log_{M} n_{1}\cdots n_{k}n_{k+1}}{-\log_{M}\frac{c_1c_2\cdots c_k-\xi_{k+1, 0}-\xi_{k+1, n_{k+1}}}{n_{k+1}}}=1.
	\end{equation*}

    Notice that $$N_{k+1}^{*}=n_1\cdots n_kn_{k+1}, \delta_{k}^{*}=c_1\cdots c_k-\xi_{k+1, 0}-\xi_{k+1, n_{k+1}}, n_{k+1}^{*}=n_{k+1},$$
	then
	\begin{equation*}
	\varlimsup _{k \rightarrow \infty}\frac{\log_{M} N_{k}^{*}}{\log_{M}N_{k}^{*}-\log_{M}N_{k-1}^{*}\delta_{k-1}^{*}}=\varlimsup _{k \rightarrow \infty}\frac{\log_{M} N_{k+1}^{*}}{-\log_{M}\frac{\delta_{k}^{*}}{n_{k+1}^{*}}}=1,
	\end{equation*}
which implies that
	\begin{equation*}
	\varlimsup _{k \rightarrow \infty}(N_{k-1}^{*}\delta_{k-1}^{*})^{\frac{1}{\log_{M}N_{k}^{*}}}=1.
	\end{equation*}
	
	Since $M^{i_{k}}\le n_{k}^{*}<M^{i_{k}+1}$, $m_{k}=i_{1}+\cdots+i_{k}$, we have $N_{k}^{*}=n_{1}^{*}\cdots n_{k}^{*}<M^{m_{k}+k}\le M^{2m_{k}}$,
	then $\log_{M}N_{k}^{*}\le 2m_{k}$. Notice that $m_k\ge 2$, then we have $m_{k}-1\ge \frac{m_{k}}{2}\ge \frac{\log_{M}N_{k}^{*}}{4}$. By Lemma \ref{lem5}, we have
	\begin{equation*}
	\begin{aligned}
	\varlimsup_{k \rightarrow \infty}(l(H_{m_{k}-1}))^{\frac{1}{m_{k}-1}}
	&\ge \varlimsup_{k \rightarrow \infty}[(1-\frac{2\chi}{M})(N_{k-1}^{*}\delta_{k-1}^{*})]^{\frac{1}{m_{k}-1}}\\
	&\ge \varlimsup_{k \rightarrow \infty}[(1-\frac{2\chi}{M})(N_{k-1}^{*}\delta_{k-1}^{*})]^{\frac{4}{\log_{M}N_{k}^{*}}}\\
	&=1.
	\end{aligned}
	\end{equation*}

	Notice that
	$$\varlimsup_{k \rightarrow \infty}(l(H_{m_{k}-1}))^{\frac{1}{m_{k}-1}}\le\varlimsup_{k \rightarrow \infty} (l(H_{m_{k}-1}))^{0}=1,$$
	then
	$$\varlimsup_{k \rightarrow \infty}(l(H_{m_{k}-1}))^{\frac{1}{m_{k}-1}}=1.$$
	Therefore, there exists a subsequence $\left\{a_{k}\right\}_{k\ge0}$ of $\left\{m_{k}-1\right\}_{k\ge0}$, such that
	$$\lim_{k \rightarrow\infty}(l(H_{a_{k}}))^{\frac{1}{a_{k}}}=1.$$
	
	(2) Since $l(H_{m})\le \prod_{j=1}^{m}(1-\gamma(j))$ for any  $m\ge1$,  we have $l(H_{a_{k}})\le \prod_{j=1}^{a_{k}}(1-\gamma(j))$, thus
	$$(l(H_{a_{k}}))^{\frac{1}{a_{k}}}\le (\prod_{j=1}^{a_{k}}(1-\gamma(j)))^{\frac{1}{a_{k}}}\le 1-{\frac{1}{a_{k}}}\sum_{j=1}^{a_{k}}\gamma(j).$$
	Let $k\to\infty$, combining the result of (1), we obtain that $1\le 1-\lim_{k \rightarrow\infty}{\frac{1}{a_{k}}}\sum_{j=1}^{a_{k}}\gamma(j)$, which implies that
	\begin{equation}\label{54}
	\lim_{k\rightarrow\infty}{\frac{1}{a_{k}}}\sum_{j=1}^{a_{k}}\gamma(j)=0.
	\end{equation}

    Notice that
	$$\frac{a_{k}-\#S_{\varepsilon}(a_{k})}{a_{k}}=\frac{\#\{i:1\le i\le a_{k}, \gamma(i)>\varepsilon\}}{a_{k}}\le \frac{1}{a_{k}\varepsilon}\sum_{i=1}^{a_{k}}\gamma(i),$$
	by the equality (\ref{54}), we get
	\begin{equation*}
	\lim_{k \rightarrow\infty}\frac{\# S_{\varepsilon}(a_{k})}{a_{k}}=1.
	\end{equation*}
	
	(3) For any $p\in(0,1]$, by the Jensen's inequality, we have
	\begin{equation*}
	\frac{1}{a_{k}}\sum_{j=1}^{a_{k}}(\gamma(j))^{p}\le(\frac{1}{a_{k}}\sum_{j=1}^{a_{k}}\gamma(j))^{p},
	\end{equation*}
	combining the equality (\ref{54}), we have
	\begin{equation}\label{55}
	\lim_{k \rightarrow\infty}\frac{1}{a_{k}}\sum_{j=1}^{a_{k}}(\gamma(j))^{p}=0.
	\end{equation}

	Since $\log(1-x)\ge -2x$ for any $x\in(0,\frac{1}{2})$, we get
	\begin{equation*}
	\begin{aligned}
	\log(\prod_{j\in S_{\varepsilon}(a_{k})}(1-(\gamma(j))^{p}))^{\frac{1}{a_{k}}}
	&=\frac{1}{a_{k}}\sum_{j\in S_{\varepsilon}(a_{k})}\log(1-(\gamma(j))^{p})\\
	&\ge \frac{-2}{a_{k}}\sum_{j\in S_{\varepsilon}(a_{k})}(\gamma(j))^{p}\\
	&\ge \frac{-2}{a_{k}}\sum_{j=1}^{a_{k}}(\gamma(j))^{p}\\
	\end{aligned}
	\end{equation*}
for sufficiently small $\varepsilon\in(0,1)$.
	Together with (\ref{55}), we obtain that
	\begin{equation*}
	\lim_{k \rightarrow\infty}(\prod_{j\in S_{\varepsilon}(a_{k})}(1-(\gamma(j))^{p}))^{\frac{1}{a_{k}}}=1.
	\end{equation*}
\end{proof}

\bigskip
\subsection{The estimate of $\mu_{d}$}
 Let $E=E(I_{0},\left\{n_{k}\right\},\left\{c_{k}\right\},\left\{\xi_{k,l}\right\})$ satisfies the condition of Theorem \ref{thm} with $\operatorname{dim}_{P}E=1$, $f$ be a \text{\rm1}-dimensional quasisymmetric mapping, $\left\{H_{m}\right\}_{m\ge0}$ be the length decreasing sequence in Lemma \ref{lem4} and  $\{a_k\}_{k\ge 1}$ be the sequence in Lemma \ref{lem6}. For using Lemma \ref{lem2} to prove Theorem \ref{thm}, we first estimate $\mu_d(J)$ for any branch $J$ of $f(H_{a_k})$ for any $k\ge 1$.

\begin{proposition}\label{prop}
	For any $k\ge 1$, any branch of $f(H_{a_k})$, denoted by $J$, there is a positive constant $C($independent of $J)$ satisfying $\mu_{d}(J)\le C\left|J\right|^{d}$ for any $d\in(0,1)$.
\end{proposition}
\begin{proof}
	For any $d\in(0,1)$ and $k\ge 1$, let $J=J_{a_k}$ be a branch of $f(H_{a_{k}})$. For any $0\le j\le a_{k}-1$, let $J_{j}$ be a branch of $f(H_{j})$ satisfying
	$$J=J_{a_{k}}\subset J_{a_{k}-1}\subset\cdots\subset J_{1}\subset J_{0}=f(I_{0}^{*}).$$
	With loss of generality, suppose $J_{0}=I_{0}^{*}=[0,1]$. By the definition of $\mu_{d}$, it is obvious that
	$$\mu_{d}(J_{a_{k}})=\frac{\left|J_{a_{k}}\right|^{d}}{\sum_{i=1}^{N(J_{a_{k}}-1)}\left|J_{a_{k}-1,i}\right|}\cdot\frac{\left|J_{a_{k}-1}\right|^{d}}{\sum_{i=1}^{N(J_{a_{k}}-2)}\left|J_{a_{k}-2,i}\right|}\cdots \frac{\left|J_{1}\right|^{d}}{\sum_{i=1}^{N(J_{0})}\left|J_{0,i}\right|}\left|J_{0}\right|^{d}.$$
	Thus
	\begin{equation}\label{56}
	\frac{\mu_{d}(J_{a_{k}})}{\left|J_{a_{k}}\right|^{d}}=\prod_{j=0}^{a_{k}-1}\frac{\left|J_{j}\right|^{d}}{\sum_{i=1}^{N(J_{j})}\left|J_{j,i}\right|^{d}}.
	\end{equation}
	
	We start to estimate $\frac{\sum_{i=1}^{N(J_{j})}\left|J_{j,i}\right|^{d}}{\left|J_{j}\right|^{d}}\ $ for any $0\le j\le a_{k}-1$. For any $k\ge 0$, $\sigma\in D_{k}$, let $L_{\sigma*0}=[\min(I_{\sigma}^{*}),\min(I_{\sigma*1}^{*}))$, $L_{\sigma*n_{k}^{*}}=(\max(I_{\sigma*n_{k}^{*}}^{*}),\max(I_{\sigma}^{*})]$. For any $0\le j\le a_{k}-1$, let $J_{j,1},J_{j,2},\cdots, J_{j,N(_{J_{j}})}\subset J_{j}\subset f(H_{j})$ be branches of $f(H_{j+1})$ located from left to right in $J_{j}$ and $G_{j,1},\cdots,G_{j,N(J_{j})-1}$ be gaps between $J_{j,1},J_{j,2},\cdots, J_{j,N(_{J_{j}})}$. Let
	$$I_{j}=f^{-1}(J_{j}),\ I_{j,l}=f^{-1}(J_{j,l})\ (\forall 1\le l\le N(J_{j}));$$
	$$L_{j,l}=f^{-1}(G_{j,l})\ (\forall 1\le l\le N(J_{j})-1).$$
	It is obvious that
	\begin{enumerate}
		\item[\textup{(1)}] $I_{j}=f^{-1}(J_{j})$ is a branch of $H_{j}$ and any $I_{j,l}\subset I_{j}$ is a branch of $H_{j+1}$;
		\item[\textup{(2)}] $L_{j,1},\cdots,L_{j,N(J_{j})-1}$ are gaps between $I_{j,1},I_{j,2},\cdots,I_{j,N(J_{j})}$.
	\end{enumerate}
    By Lemma \ref{lem3}, we have
    \begin{equation}\label{57}
    \frac{\left|J_{j,l}\right|}{\left|J_{j}\right|}\ge \beta(\frac{\left|I_{j,l}\right|}{\left|I_{j}\right|})^{q}\ge \beta(\lambda(j+1))^{q}\ge (4\chi^{2})^{-q}\beta(\Lambda(j+1))^{q},
    \end{equation}
    \begin{equation}\label{58}
     \frac{\left|G_{j,l}\right|}{\left|J_{j}\right|}\le 4(\frac{\left|L_{j,l}\right|}{\left|I_{j}\right|})^{p}\le 4(\Gamma(j+1))^{p}\le4(2\chi^{2}\gamma(j+1)))^{p}\le 8\chi^{2}(\gamma(j+1))^{p}.
    \end{equation}
    It follows from (\ref{57}) that
    \begin{equation}\label{59}
    \frac{\sum_{i=1}^{N(J_{j})}\left|J_{j,i}\right|^{d}}{\left|J_{j}\right|^{d}}
    \ge (4\chi^{2})^{-dq}\beta^{d}(\Lambda(j+1))^{dq}.
    \end{equation}

    \bigskip
    Next we do another estimation of $\frac{\sum_{i=1}^{N(J_{j})}\left|J_{j,i}\right|^{d}}{\left|J_{j}\right|^{d}}$ for any $(j+1)\in S_{\varepsilon}(a_{k})$ with sufficiently small $\varepsilon\in(0,1)$.

    Let $k\ge 1$ satisfies $m_{k-1}\le j<m_{k}$ and $I_{j}\subset I_{\sigma}^{\ast}(\sigma \in D_{k-1})$, together $N(J_{j})\le M^{2}$ with (\ref{58}), we have
   \begin{equation*}
   \begin{aligned}
   \frac{\sum_{i=1}^{N(J_{j})}\left|J_{j,i}\right|}{\left|J_{j}\right|}
   &\ge \frac{\left|J_{j}\right|-\sum_{l=1}^{N(J_{j})-1}\left|G_{j,l}\right|-\left|f(L_{\sigma*0})\right|-|f(L_{\sigma*n_{k}^{*}})|}{\left|J_{j}\right|}\\
   &\ge 1-8M^{2}\chi^{2}(\gamma(j+1))^{p}-4((\frac{\xi_{k,0}^{*}}{\left|I_{j}\right|})^{p}+(\frac{\xi_{k,n_{k}}^{*}}{\left|I_{j}\right|})^{p}).\\
   \end{aligned}
   \end{equation*}
   Notice that $\xi_{k,0}^{*}+\xi_{k,n_{k}}^{*}\le \underline{\alpha}_{k}^{*}\le\overline{\alpha}_{k}^{*},$ then
   \begin{equation}\label{510}
   \begin{aligned}
   \frac{\sum_{i=1}^{N(J_{j})}\left|J_{j,i}\right|}{\left|J_{j}\right|}
   &\ge 1-8M^{2}\chi^{2}(\gamma(j+1))^{p}-8(\frac{\left|\overline{\alpha}_{k}^{*}\right|}{\left|I_{j}\right|})^{p}\\
   &\ge 1-8M^{2}\chi^{2}(\gamma(j+1))^{p}-8(\Gamma(j+1))^{p}\\
   &\ge 1-8M^{2}\chi^{2}(\gamma(j+1))^{p}-8(2\chi^{2}\gamma(j+1))^{p}\\
   &\ge 1-8\chi^{2}(M^{2}+2)(\gamma(j+1))^{p},
   \end{aligned}
   \end{equation}
  which implies that
   \begin{equation}\label{511}
   \frac{\sum_{i=1}^{N(J_{j})}\left|J_{j,i}\right|^{d}}{\left|J_{j}\right|^{d}}
   \ge \frac{\sum_{i=1}^{N(J_{j})}\left|J_{j,i}\right|^{d}}{(\sum_{i=1}^{N(J_{j})}\left|J_{j,i}\right|)^{d}}\times(1-8\chi^{2}(M^{2}+2)(\gamma(j+1))^{p})^{d}.
   \end{equation}

   For any $1\le l\le N(J_{j})$, if $(j+1)\in S_{\varepsilon}(a_{k})$ and $\varepsilon\in(0,\frac{1}{2M^{2}\chi^{2}})$, by  the construction and properties of $\{H_m\}_{m\ge0}$, we have
   \begin{equation}\label{512}
   \begin{aligned}
   \frac{\left|I_{j,l}\right|}{\left|I_{j}\right|}
   &\ge
   \frac{\left|I_{j}\right|-\sum_{l=1}^{N(J_{j})-1}\left|L_{j,l}\right|-\xi_{k,0}^{*}-\xi_{k, n_{k}}^{*}}{2M^{2}\chi\left|I_{j}\right|}\\
   &\ge
   \frac{\left|I_{j}\right|-(M^{2}-1)\overline{\alpha}_{k}^{*}-\overline{\alpha}_{k}^{*}}{2\chi M^{2}\left|I_{j}\right|}\\
   &\ge
   \frac{1}{2\chi M^{2}}(1-M^{2}\Gamma(j+1))\\
   &\ge \frac{1-2\chi^{2}M^{2}\varepsilon}{2\chi M^{2}}.
   \end{aligned}
   \end{equation}
   By Lemma \ref{lem3}, we obtain that
   \begin{equation*}
   1>\frac{\left|J_{j,l}\right|}{\left|J_{j}\right|}\ge\beta(\frac{1-2\chi^{2}M^{2}\varepsilon}{2\chi M^{2}})^{q},
   \end{equation*}
  which implies that
   $$1\ge
   \frac{\left|J_{j,l}\right|}{\max_{1\le i\le N(J_{j})}\left|J_{j,i}\right|}
   \ge\beta(\frac{1-2\chi^{2}M^{2}\varepsilon}{2\chi M^{2}})^{q}.$$
   Notice that for any $d\in(0,1)$ and $x_{1},\cdots,x_{k}\in(0,1]$,
   $$\frac{1+x_{1}^{d}+\cdots+x_{k}^{d}}{(1+x_{1}+\cdots+x_{k})^{d}}\ge(1+\max\left\{x_{1},\cdots,x_{k}\right\})^{1-d},$$
   thus
   \begin{equation}\label{513}
   \begin{aligned}
   \frac{\sum_{i=1}^{N(J_{j})}\left|J_{j,i}\right|^{d}}{(\sum_{i=1}^{N(J_{j})}\left|J_{j,i}\right|)^{d}}
   &=(\frac{\sum_{i=1}^{N(J_{j})}\left|J_{j,i}\right|^{d})/(\max_{1\le i\le N(J_j)}\left|J_{j,i}\right|)^{d}}{(\sum_{i=1}^{N(J_{j})}\left|J_{j,i}\right|)^{d}/(\max_{1\le i\le N(J_j)}\left|J_{j,i}\right|)^{d}}\\
   &\ge (1+\beta(\frac{1-2\chi^{2}M^{2}\varepsilon}{2\chi M^{2}})^{q})^{1-d}\\
   &\triangleq\kappa>1.
   \end{aligned}
   \end{equation}
   Combining (\ref{511}) with (\ref{513}), we obtain that
   $$\frac{\sum_{i=1}^{N(J_{j})}\left|J_{j,i}\right|^{d}}{\left|J_{j}\right|^{d}}
   \ge \kappa\times(1-8\chi^{2}(M^{2}+2)(\gamma(j+1))^{p})^{d}$$
   for any $\varepsilon\in (0,\frac{1}{2M^2\chi^2})$ and $(j+1)\in S_{\varepsilon}(a_{k})$.

   Let $M_{1}$ be an integer satisfying $M_{1}>8\chi^{2}(M^{2}+2)$. Notice that $1-M_{1}x>(1-x)^{M_{1}}$ if $x>0$ is sufficiently small, then if $\varepsilon>0$ is sufficiently small, for $(j+1)\in S_{\varepsilon}(a_{k})$, we have
   \begin{equation}\label{514}
   \frac{\sum_{i=1}^{N(J_{j})}\left|J_{j,i}\right|^{d}}{\left|J_{j}\right|^{d}}
   \ge \kappa\times(1-(\gamma(j+1))^{p})^{M_{1}d}.
   \end{equation}
   For sufficiently small $\varepsilon>0$, combining (\ref{59}) with (\ref{514}), notice that (\ref{52}) holds, we have
   \begin{equation*}
\begin{aligned}
  \prod_{j=0}^{a_k-1}\frac{\sum_{i=1}^{N(J_j)}|J_{j,i}|^d}{|J_{j,i}|^d}
     &\ge \prod_{0\le j\le a_k-1,(j+1)\notin S_{\varepsilon}(a_{k})}(4\chi^2)^{-dq}\beta^d(\Lambda(j+1))^{dq}\\&\times \prod_{0\le j\le a_k-1,(j+1)\in S_{\varepsilon}(a_{k})}\kappa(1-(\gamma(j+1))^p)^{M_1d}\\
     &= \prod_{0\le j\le a_k-1,(j+1)\notin S_{\varepsilon}(a_{k})}(4\chi^2)^{-dq}\beta^d\times\prod_{0\le j\le a_k-1,(j+1)\notin S_{\varepsilon}(a_{k})}(\Lambda(j+1))^{dq}\\
     &\times \prod_{0\le j\le a_k-1,(j+1)\in S_{\varepsilon}(a_{k})}\kappa(1-(\gamma(j+1))^p)^{M_1d}\\
     & \ge {((4M^{2}\chi^2)^{-dq}\beta^d)}^{a_k-\#S_{\varepsilon}(a_{k})}\times{(l(H_{a_k}))}^{dq}\\ &\times \prod_{0\le j\le a_k-1,(j+1)\in S_{\varepsilon}(a_{k})}\kappa(1-(\gamma(j+1))^p)^{M_1d}.
\end{aligned}
\end{equation*}
By (2) of Lemma \ref{lem6}, we have

\begin{equation}\label{5.15}
\lim_{k\to\infty}((4M^{2}\chi^2)^{-dq}\beta^d)^{\frac{a_k-\#S_{\varepsilon}(a_{k})}{a_k}}=1.
\end{equation}
By  (1) of Lemma \ref{lem6}, we have
\begin{equation}\label{5.16}
  \lim_{k\to\infty}(l(H_{a_k}))^{\frac{dq}{a_k}}=1.
\end{equation}
Notice that
\begin{equation*}
 \begin{aligned}
 & \prod_{0\le j\le a_k-1,(j+1)\in S_{\varepsilon}(a_{k})}\kappa(1-(\gamma(j+1))^p)^{M_1d}\\&=\kappa^{\#S_{\varepsilon}(a_{k})}(\prod_{0\le j\le a_k-1,(j+1)\in S_{\varepsilon}(a_{k})}(1-(\gamma(j+1))^p))^{M_1d},
    \end{aligned}
\end{equation*}
combining (2) and (3) of  Lemma \ref{lem6}, we have
\begin{equation}\label{5.17}
 \begin{aligned}
  &\lim_{k\to\infty}(\prod_{0\le j\le a_k-1,(j+1)\in S_{\varepsilon}(a_{k})}\kappa(1-(\gamma(j+1))^p)^{M_1d})^{\frac{1}{a_k}}\\&=\lim_{k\to\infty}\kappa^{\frac{\#S_{\varepsilon}(a_{k})}{a_k}}(\prod_{0\le j\le a_k-1,(j+1)\in S_{\varepsilon}(a_{k})}(1-(\gamma(j+1))^p))^{\frac{M_1d}{a_k}}\\&=\kappa.\
   \end{aligned}
\end{equation}
Together with (\ref{5.15}), (\ref{5.16}) and (\ref{5.17}), we obtain that
\begin{equation*}
  \varliminf_{k\to\infty}(\prod_{j=0}^{a_k-1}\frac{\sum_{i=1}^{N(J_j)}|J_{j,i}|^d}{|J_{j,i}|^d})^{\frac{1}{a_k}}\ge \kappa>1,
\end{equation*}
which implies that
\begin{equation}\label{5.18}
  \varliminf_{k\to\infty}\prod_{j=0}^{a_k-1}\frac{\sum_{i=1}^{N(J_j)}|J_{j,i}|^d}{|J_{j,i}|^d}=\infty.
\end{equation}

   Combining (\ref{5.18}) with the equality (\ref{56}), we obtain that there is a constant $C>0$, such that
   \begin{equation*}
   \mu_{d}(J)\le C\left|J\right|^{d}.
   \end{equation*}
\end{proof}
\bigskip

\subsection{The proof of Theorem \ref{thm}}
Now we start to finish the proof of Theorem \ref{thm}. Let $\{a_k\}_{k\ge 1}$ be the sequence in Lemma \ref{lem6}.  For any $x\in f(E)$, since $h_{x}(r)=\left|f^{-1}(B(x,r))\right|$ is a continuous mapping and $\lim\limits_{r\to 0}h_x(r)=0$, there exists a sequence  $\left\{r_{k}\right\}_{k\ge1}$ satisfying
\begin{equation*}
\min_{I\in\mathcal H_{a_{k}}}\left|I\right|\le h_{x}(r_{k})<\min_{I\in\mathcal H_{a_{k}-1}}\left|I\right|.
\end{equation*}
Then $f^{-1}(B(x,r_{k}))$ meets at most two branches of $H_{a_{k}-1}$, thus it meets at most $2M^{2}$ branches of $H_{a_{k}}$, and $B(x,r_{{k}})$ meets at most $2M^{2}$ branches of $f(H_{a_{k}})$ .

Let $R_{1},R_{2},\cdots,R_{l}\ (1\le l\le 2M^{2})$ be the branches of $f(H_{a_{k}})$ meetting $B(x,r_{k})$, then
\begin{equation*}
B(x,r_{k})\cap f(E)\subset R_{1}\cup R_{2}\cup\cdots\cup R_{l}.
\end{equation*}
By Proposition \ref{prop}, we get
\begin{equation}\label{mu11}
\mu_{d}(B(x,r_{k}))=\mu_{d}(B(x,r_{k})\cap f(E))\le\sum_{j=1}^{l}\mu_{d}(R_{j})\le C\sum_{j=1}^{l}\left|R_{j}\right|^{d}.
\end{equation}

Notice that
\begin{equation*}
\min_{I\in\mathcal H_{a_{k}}}\left|I\right|\le \left|f^{-1}(B(x,r_{k}))\right|,\ \max_{I\in\mathcal H_{a_{k}}}\left|I\right|\le 2\chi\min_{I\in H_{a_{k}}}\left|I\right|,
\end{equation*}
then for any $1\le j \le l$,
\begin{equation*}
\left|f^{-1}(R_{j})\right|\le\max_{I\in\mathcal H_{a_{k}}}\left|I\right|\le 2\chi\min_{I\in\mathcal H_{a_{k}}}\left|I\right|\le 2\chi\left|f^{-1}(B(x,r_{k}))\right|.
\end{equation*}
Notice that $B(x,r_{k})\cap R_{j}\neq\emptyset$ for any $1\le j \le l$ ,  then
\begin{equation*}
f^{-1}(R_{j})\subset 6\chi f^{-1}(B(x,r_{k})),
\end{equation*}
where for any $\rho>0$, $\rho I$ denotes the interval with the same center of the  interval $I$, and $|\rho I|=\rho|I|$.
By Lemma \ref{lem3}, there is a constant $K_{6\chi}$, such that
\begin{equation}\label{lt}
|R_j|=\left|f(f^{-1}(R_{j}))\right|\le\left|f(6\chi f^{-1}(B(x,r_{k})))\right|\le K_{6\chi}\left|B(x,r_{k})\right|= 2K_{6\chi}r_{k}.
\end{equation}

Combing (\ref{mu11}) and (\ref{lt}), we have
\begin{equation*}
\begin{aligned}
\mu_{d}(B(x,r_{k}))
&\le C\sum_{j=1}^{l}\left|R_{j}\right|^{d}\\
&\le C\cdot 2M^{2}(2K_{6\chi}r_{k})^{d}\\
&=4K_{6\chi}^{d}M^{2}C(r_{k})^{d}\\
&\triangleq C_{1}(r_{k})^{d}.
\end{aligned}
\end{equation*}
Since $\lim\limits_{k\to \infty}r_{k}=0$, then for any $x\in f(E)$,
\begin{equation}\label{516}
\varliminf_{r \rightarrow 0}\frac{\mu_{d}(B(x,r))}{r^{d}}\le C_{1}.
\end{equation}

By Lemma \ref{lem2} and (\ref{516}), we obtain that $\operatorname{dim}_{P}f(E)\ge d$. Notice that $d\in(0,1)$ is arbitrary, then we have
\begin{equation*}
\operatorname{dim}_{P}f(E)\ge1,
\end{equation*}
which implies that $\operatorname{dim}_{P}f(E)=1$.

\end{document}